\documentclass{amsart}
\usepackage{amssymb}
\usepackage{graphics}
\usepackage{latexsym}
\usepackage{amsmath}
\usepackage{amssymb}
\usepackage{amscd}
\usepackage{multicol}
\usepackage[cmtip,matrix,arrow]{xy}

\newtheorem{thm}{Theorem}[section]
\newtheorem{prop}[thm]{Proposition}
\newtheorem{cor}[thm]{Corollary}
\newtheorem{lem}[thm]{Lemma}

\theoremstyle{definition}

\newtheorem{ex}[thm]{Example}

\theoremstyle{remark}
\newtheorem{rem}[thm]{Remark}

\newcommand{\RR}{\mathbb R}
\newcommand{\ZZ}{\mathbb Z}
\newcommand{\QQ}{\mathbb Q}

\newcommand{\rk}{\mathrm{rank}\, }

\newcommand{\Ad}{\operatorname{Ad}}

\newcommand{\mfk}{\mathfrak{k}}

\newcommand{\mfg}{\mathfrak{g}}
\newcommand{\mfh}{\mathfrak{h}}
\newcommand{\mft}{\mathfrak{t}}

\newcommand{\depth}{\operatorname{depth}}

\newcommand{\Mmax}{M_{\max}}

\newcommand{\h}{\mathfrak{h}}

\begin{document}
\address[O.~Goertsches]{Fachbereich Mathematik\\  Universit\"at Hamburg\\  Germany}
\address[A.-L.~Mare]{Department of Mathematics and Statistics\\ University of Regina\\ Canada}

\email[]{oliver.goertsches@math.uni-hamburg.de}

\email[]{mareal@math.uregina.ca}

\title{Equivariant cohomology of cohomogeneity one actions }

\author{Oliver Goertsches}
\author{Augustin-Liviu Mare}

\begin{abstract} We show that if $G\times M \to M$ is a cohomogeneity one action of a compact connected Lie group
$G$  on a compact connected manifold $M$ then $H^*_G(M)$ is a Cohen-Macaulay module over $H^*(BG)$.
Moreover, this module is free if and only if the rank of at least one isotropy group is equal to $\rk G$. 
We deduce as corollaries several results concerning the usual (de Rham) cohomology of $M$, such as the following obstruction to the existence of a cohomogeneity one action: if $M$ admits a cohomogeneity one action, then $\chi(M)>0$ if and only if 
 $H^{\rm odd}(M)=\{0\}$.
\end{abstract}
\maketitle

\section{Introduction}\label{one}
Let $G$ be a compact connected Lie group which acts on a compact connected manifold $M$,
the cohomogeneity of the action being equal to one; this means that there exists a $G$-orbit whose
codimension in $M$ is equal to one.  
For such group actions, we investigate the corresponding equivariant cohomology  $H^*_G(M)$ (the coefficient ring
will always be $\RR$). 
We are especially interested in the natural $H^*(BG)$-module structure of this space.
The first natural question concerning this module is whether it is free, in other words, whether the $G$-action is
equivariantly formal. One can easily find examples  which show that the answer is in general negative.
Instead of being free, we may also wonder whether the above-mentioned module satisfies the (weaker) requirement
of being Cohen-Macaulay. It turns out that the answer is in our context always positive: this is the main
result of our paper. Before stating it, we mention that the relevance of the Cohen-Macaulay condition in equivariant
cohomology was for the first time noticed by Bredon  \cite{Bredon}, inspired by Atiyah \cite{At}, who had previously 
used this notion in
equivariant $K$-theory. It has also attracted attention in the theory of equivariant cohomology of finite group actions, see e.g.~\cite{Duflot}. More recently, group actions whose equivariant cohomology satisfies this requirement have been investigated in  \cite{FranzPuppe2003}, \cite{GT'}, and \cite{GR}. We adopt the terminology already used in those papers: 
 if a group $G$ acts on a space $M$ in such a way
 that $H^*_G(M)$ is a Cohen-Macaulay $H^*(BG)$-module, we simply say that the $G$-action is Cohen-Macaulay.
  
 \begin{thm}\label{thm:cohom1eqformal}
Any cohomogeneity one action of a compact connected Lie group on a compact connected manifold is Cohen-Macaulay.  
\end{thm}

Concretely, if the group action is $G\times M \to M$, then the (Krull) dimension and the depth of $H^*_G(M)$ over
$H^*(BG)$ are equal. In fact, we can  say exactly what the value of these two numbers is: the highest
rank of a $G$-isotropy group.

To put our theorem into perspective, we mention the following result, which has been proved in \cite{GR}:
an action of a compact connected Lie group on a compact manifold with the property that all isotropy groups
have the same rank is Cohen-Macaulay. 
Consequently, if the $G$-action is  transitive, then it is Cohen-Macaulay (see also Proposition \ref{lem:cohomrank} 
and Remark \ref{assertion} below). We deduce:

\begin{cor}\label{<2}
Any action of a compact connected Lie group on a compact connected manifold whose cohomogeneity is zero or
 one is Cohen-Macaulay.  
\end{cor}

We also note that actions of cohomogeneity two or larger are not necessarily Cohen-Macaulay: examples already appear in the classification of $T^2$-actions on $4$-manifolds by Orlik and Raymond \cite{Or-Ra}, see Example
\ref{exor} in this paper.

In general, a group action is equivariantly formal
if and only if it is  Cohen-Macaulay and the rank of at least one isotropy group is maximal, i.e.~equal to the rank of
the acting group (cf.~\cite{GR}, see also
Proposition \ref{gr} below). This immediately implies the following characterization of equivariant formality 
for cohomogeneity one actions:  

\begin{cor}\label{cor:eqf} A cohomogeneity one action of a compact connected Lie group on a compact connected manifold is equivariantly formal if and only if the  rank of at least one isotropy group is maximal.
\end{cor}


Corollary \ref{cor:eqf} shows that
the cohomogeneity one action $G\times M \to M$ is equivariantly formal whenever $M$ satisfies the 
purely topological condition $\chi(M)>0$
 (indeed, it is known that this inequality implies   the condition on the rank of the isotropy groups in Corollary \ref{cor:eqf}). Extensive lists of cohomogeneity one actions on manifolds with positive Euler characteristic can be found 
for instance  in \cite{AlekseevskyPodesta} and \cite{Fr}.
The above observation will be used to obtain  the following obstruction to the existence of a group action on $M$ of
cohomogeneity zero or one.
\begin{cor}\label{compact} If a compact manifold $M$ admits a cohomogeneity one action of a compact Lie group, then we have $\chi(M)>0$ if and only if $H^{\rm odd}(M)=\{0\}$.
\end{cor}
This topic is addressed in Subsection \ref{subsec:ec} below. We also mention that, if $M$ is as in
 Corollary \ref{compact}, then $\chi(M)>0$ implies that $\pi_1(M)$ is finite, see Lemma \ref{pi1finite}. By classical results of Hopf and Samelson \cite{Ho-Sa}, respectively Borel \cite{Bo},
 the fact that $\chi(M)>0$ implies both $H^{\rm odd}(M)=\{0\}$ and the finiteness of $\pi_1(M)$ holds true also in the case when $M$ admits an action 
of a compact Lie group which is transitive, i.e.~of cohomogeneity equal to zero. 
This shows, for example, that there is no compact connected Lie group action with cohomogeneity zero or one on a compact manifold with the rational homology type of the connected sum $(S^1\times S^3)\# (S^2 \times S^2)$. 
However, the 4-manifold $R(1,0)$ of Orlik and Raymond \cite{Or-Ra} mentioned in Example \ref{exor}  is homeomorphic to this connected sum
and has a $T^2$-action of cohomogeneity two. Thus, the equivalence of $\chi(M)>0$ and $H^{\rm odd}(M)=\{0\}$ 
holds no longer for group actions with cohomogeneity greater than one. 

It should be noted that Corollary \ref{compact} is not a new result, as it follows also from a result of Grove and Halperin \cite{GH} about the rational homotopy of cohomogeneity one actions, see Remark \ref{rii} below. 

The situation when $M$ is odd-dimensional is discussed in Subsection \ref{subsec:misod}.
In this case, equivariant formality is equivalent to $\rk H = \rk G$, where $H$ denotes a regular isotropy of the $G$-action.
We obtain a relation involving $\dim H^*(M)$,   the Euler characteristic of $G/H$ and 
the Weyl group $W$ of the cohomogeneity one action: see Corollary \ref{cor:oddd}. 
This will enable us to obtain some results for cohomogeneity one actions
on odd-dimensional rational homology spheres.

Finally, in Subsection \ref{torsfre} we show that for a cohomogeneity one action
$G\times M \to M$, the $H^*(BG)$-module $H^*_G(M)$ is torsion free if and only if it
is free. We note that the latter equivalence is in general not true for arbitrary group actions:
this topic is investigated in  \cite{Al} and \cite{Fr-Pu2}. 

\noindent{\bf Acknowledgement.} We would like to thank Wolfgang Ziller for some valuable comments
on a previous version of this paper.

\section{Equivariantly formal and Cohen-Macaulay group actions}\label{sec:lifting}

Let $G\times M \to M$ be a differentiable group action, where both the Lie group $G$ and the
manifold $M$ are compact. The equivariant cohomology ring is defined in the usual way,
by using the classifying principal $G$ bundle $EG\to BG$, as follows:
$H^*_G(M)=H^*(EG \times_G M)$. It has a canonical structure of $H^*(BG)$-algebra, induced 
by the ring homomorphism $\pi^*:H^*(BG)\to H^*_G(M)$, where $\pi : EG \times_G M \to BG$ is the canonical map 
(cf.~e.g.~ \cite[Ch.~III]{Hs} or \cite[Appendix C]{GGK}).
We say that the group action is {\it equivariantly formal} if $H^*_G(M)$  regarded as an 
$H^*(BG)$-module is free. 

There is also a relative notion of equivariant cohomology: if $N$ is a $G$-invariant submanifold of $M$, then we define $H^*_G(M,N)=H^*(EG\times_G M, EG\times_G N)$. This cohomology carries again an $H^*(BG)$-module structure, this time induced by the map $H^*(BG)\to H^*_G(M)$ and the cup product $H^*_G(M)\times H^*_G(M,N)\to H^*_G(M,N)$; see \cite[p.~178]{tomDieck}. 

\subsection{Criteria for equivariant formality}
The following result is known. We state it for future reference and sketch a proof for the reader's convenience. 
It also involves $T$, which is a maximal torus of $ G$.  

\begin{prop}\label{abo} If a compact connected Lie group $G$ acts on a compact manifold $M$, then the following statements are equivalent:

(a) The $G$-action on $M$ is equivariantly formal.

(b) We have 
$H^*_G(M) \simeq H^*(M)\otimes H^*(BG)$ by an 
isomorphism of
$H^*(BG)$-modules.

(c) The  homomorphism $i^*: H^*_G(M)\to H^*(M)$ induced by the canonical inclusion
$i: M \to EG\times_G M$ is surjective. In this case,
it automatically descends to a 
ring isomorphism ${\mathbb R} \otimes_{H^*(BG)}H^*_G(M) \to H^*(M)$.

(d) The $T$-action on $M$ induced by restriction is equivariantly formal.

(e) $\dim H^*(M)=\dim H^*(M^T)$, where $M^T$ is the $T$-fixed point set.

\end{prop}

\begin{proof} A key ingredient of the proof is the Leray-Serre spectral sequence of the
bundle $M \stackrel{i}{\to}  EG\times_G M \stackrel{\pi}{\to} BG$, whose $E_2$-term is $H^*(M)\otimes H^*(BG)$. Both (b) and (c) are clearly equivalent to the fact that this spectral
sequence collapses at  $E_2$. The equivalence to (a) is more involved, see for instance
\cite[Corollary (4.2.3)]{AlldayPuppe}.
The equivalence to (d) is the content of \cite[Proposition C.26]{GGK}.  For the equivalence to (e), see \cite[p.~46]{Hs}.
\end{proof} 

We also mention the following useful criterion for equivariant formality:

\begin{prop}\label{wealso}
An action of a compact connected Lie group $G$ on a compact manifold $M$ with $H^{\rm odd}(M)=\{0\}$ is automatically equivariantly formal.
The converse implication is true provided that the $T$-fixed point set $M^T$ is finite.
\end{prop}

\begin{proof} If $H^{\rm odd}(M)=\{0\}$ then the Leray-Serre spectral sequence of the bundle 
$ET \times_T M \to BT$ collapses at $E_2$, which is equal to $H^*(M)\otimes H^*(BT)$.
To prove the second assertion we note that the Borel localization theorem implies that the kernel of the natural map $H^*_T(M)\to H^*_T(M^T)$ equals the torsion submodule of $H^*_T(M)$, which vanishes in the case of an equivariantly formal action. Because $M^T$ is assumed to be finite, $H^*_T(M^T)$ is concentrated only in even degrees. Consequently, the same holds true for $H^*_T(M)$,
as well as for $H^*(M)$, due to the isomorphism $H^*_T(M)\simeq H^*(M)\otimes H^*(BT)$.
\end{proof}

\subsection{Cohen-Macaulay modules and group actions}\label{subsec:cm}
If $R$ is a graded *local Noetherian graded ring, one can associate to any finitely generated graded $R$-module $A$ its Krull dimension,
respectively depth (see e.g.~\cite[Section 1.5]{BrunsHerzog}). We always have that $\depth A\leq \dim A$, and if dimension and depth coincide, then we say that the finitely-generated graded $R$-module $A$ is a \emph{Cohen-Macaulay module} over $R$.
The following result is an effective tool frequently used in this paper:

\begin{lem}\label{ineq} If $0 \to A \to B \to C \to 0$ is a short exact sequence of finitely generated $R$-modules, then we have:
\begin{itemize}
\item[(i)]  $\dim B = {\rm max} \{\dim A, \dim C\}$;
\item[(ii)] ${\rm depth} \,  A \ge {\rm min} \{ {\rm depth}  \, B, {\rm depth } \,  C +1\}$;
\item[(iii)] ${\rm depth} \,  C \ge {\rm min} \{ {\rm depth} \, A-1, {\rm depth } \, B\}$.
\end{itemize}
\end{lem}
\begin{proof}
For (i) we refer to \cite[Section 12]{Matsumura} and for (ii) and (iii) to  \cite[Proposition 1.2.9]{BrunsHerzog}. 
\end{proof}
\begin{cor}\label{sumofCM}
If $A$ and $B$ are Cohen-Macaulay modules over $R$ of the same dimension $d$, then $A\oplus B$ is again Cohen-Macaulay of dimension $d$.
\end{cor}
\begin{proof} This follows from Lemma \ref{ineq}, applied to the exact sequence $0 \to A \to A \oplus B \to B \to 0$.
\end{proof}

Yet another useful tool will be for us the following lemma, which is the graded version of  \cite[Proposition 12, Section  IV.B]{Se}.

\begin{lem}\label{serre} Let $R$ and $S$ be two Noetherian graded *local rings and let $\varphi : R\to S$ be a homomorphism
that makes $S$ into an $R$-module which is finitely generated. If $A$ is a finitely generated $S$-module, then we have:
$$ {\rm depth}_R \, A ={\rm depth}_S \, A  \quad {\it and} \quad \dim_R A  = \dim_S A.$$
In particular, $A$ is Cohen-Macaulay as $R$-module if and only if it is Cohen-Macaulay as $S$-module.
\end{lem} 

We say that the group action $G\times M \to M$ is {\it Cohen-Macaulay} if
$H^*_G(M)$ regarded as  $H^*(BG)$-module is Cohen-Macaulay. 
The relevance of this notion for the theory of equivariant cohomology was for the first time noticed  by Bredon in \cite{Bredon}.
Other references are \cite{FranzPuppe2003}, \cite{GR}, and 
  \cite{GT'}. 

The following result gives an example of a Cohen-Macaulay action, which is important for this paper.
We first note that if $G$ is a compact Lie group and $K\subset G$ a subgroup, then there is a canonical map
$BK\to BG$ induced  by the presentations $BG = EG/G$ and $BK= EG/K$. The ring homomorphism $H^*(BG)\to H^*(BK)$
makes $H^*(BK)$ into an $H^*(BG)$-module.   
\begin{prop}\label{lem:cohomrank} Let $G$ be a compact connected Lie group
and $K\subset G$  a Lie subgroup, possibly non-connected. Then $H^*_G(G/K)=H^*(BK)$ is a Cohen-Macaulay module over $H^*(BG)$ of dimension equal to $\rk K$. 
\end{prop}
\begin{proof}
The fact that $H^*_G(G/K)= H^*(BK)$ is Cohen-Macaulay is a very special case of \cite[Corollary 4.3]{GR} because all isotropy groups of the natural $G$-action on $G/K$ have the same rank. 
To find its dimension, we consider the identity component $K_0$ of $K$ and note that also 
$H^*(BK_0)$ is a Cohen-Macaulay module over $H^*(BG)$;  by Lemma \ref{serre}, 
the dimension of $H^*(BK_0)$ over $H^*(BG)$ is equal to the (Krull) dimension  
of the (polynomial) ring $H^*(BK_0)$, which is $\rk K_0$. Let us now observe that
  $H^*(BK)=H^*(BK_0)^{K/K_0}$ is  a 
nonzero 
$H^*(BG)$-submodule of the Cohen-Macaulay module $H^*(BK_0)$, hence by \cite[Lemma 5.4]{GT'} (the graded version of \cite[Lemma 4.3]{FranzPuppe2003}),
$\dim_{H^*(BG)} H^*(BK)= \dim_{H^*(BG)}H^*(BK_0) =\rk K_0.$
\end{proof}

As a byproduct, we can now easily deduce the following result.

\begin{cor}\label{nonc} If $G$ is a (possibly non-connected) compact Lie group, then $H^*(BG)$ is a Cohen-Macaulay ring.
\end{cor}

\begin{proof} 
There exists a unitary group $U(n)$ which contains $G$ as a subgroup. By Proposition 
\ref{lem:cohomrank}, $H^*(BG)$ is a Cohen-Macaulay module over $H^*(BU(n))$. Because $H^*(BG)$ is Noetherian by a result of Venkov \cite{Ve}, we can apply Lemma \ref{serre} and conclude that $H^*(BG)$ is a Cohen-Macaulay ring.
\end{proof}

\begin{rem}\label{assertion} The assertions  in Proposition \ref{lem:cohomrank} and Corollary \ref{nonc} are not new and can also be justified
as follows. Denote by
${\mathfrak g}$  the Lie algebra of a compact Lie group $G$ and by ${\mathfrak t}$ the Lie algebra of a maximal torus $T$ in $G$,
then 
$H^*(BG)=S({\mathfrak g}^*)^G=S({\mathfrak t}^*)^{W(G)}$,
where $W(G)=N_G(T)/Z_G(T)$ is the Weyl group of $G$. (Here we have used the Chern-Weil isomorphism 
and the Chevalley restriction theorem, see e.g.~\cite[p.~311]{Milnor}, respectively \cite[Theorem 4.12]{PaTe}).
The fact that the ring $S({\mathfrak t}^*)^{W(G)}$ is Cohen-Macaulay follows from \cite[Proposition 13]{Ho-Ea} 
(see also
\cite[Corollary 6.4.6]{BrunsHerzog} or \cite[Theorem B, p.~176]{Kane}).  Proposition \ref{lem:cohomrank} 
is now a direct consequence of Lemma \ref{serre} and the well-known fact that the $G$-equivariant cohomology of a compact manifold is a finitely generated
$H^*(BG)$-module, see e.g.~\cite{Quillen}.
 \end{rem}

In general, if a group action $G\times M \to M$ is equivariantly formal, then it is Cohen-Macaulay.
The next result, which is actually Proposition 2.5  in \cite{GR}, establishes a more precise relationship between these two notions.
It also involves   
$$\Mmax:=\{p \in M \, :  \, \rk G_p = \rk G\}.$$

\begin{prop}\label{gr} {\rm (\cite{GR})} A $G$-action on $M$ is equivariantly formal 
if and only if it is Cohen-Macaulay and $\Mmax\neq \emptyset$.
\end{prop}

Note, in particular, that the $G$-action on $G/K$ mentioned in Proposition \ref{lem:cohomrank} is equivariantly
formal if and only if $\rk K = \rk G$. 

\section{Topology of transitive group actions on spheres}
The following proposition will be needed in the proof of the main result. It collects results that, in a slightly more particular situation,  were obtained by Samelson in \cite{Sa}.

\begin{prop}\label{samelson} Let $K$ be a compact Lie group, possibly non-connected, which acts transitively on
the sphere $S^{m}$, $m\ge 0$, and let $H\subset K$ be an isotropy subgroup.

(a) If $m$ is even, then $\rk H = \rk K$ and the canonical homomorphism $H^*(BK) \to H^*(BH)$ is injective.

(b) If $m$ is odd, then $\rk H = \rk K -1$ and the  canonical homomorphism $H^*(BK) \to H^*(BH)$ is surjective.
\end{prop}

\begin{proof} In the case when both $K$ and $H$ are connected, the result follows from
\cite[Satz IV]{Sa} along with \cite[Section 21, Corollaire]{Bo} and \cite[Proposition 28.2]{Bo}.
From now on, $K$ or $H$ may be non-connected. We distinguish the following three situations.

\noindent{\it Case 1:} $m \ge 2$.   Let $K_0$ be the identity component of $K$.
The induced action of $K_0$   on $S^{m}$ is also transitive, because its orbits are open and closed in the $K$-orbits.
We deduce that we can identify $K_0/(K_0\cap H)$ with $S^{m}$. Since the latter sphere is simply connected,
 the long exact homotopy sequence of the bundle $K_0\cap H \to K_0 \to S^{m}$
shows that   $K_0\cap H$ is connected. This implies that $K_0\cap H$ is equal to $H_0$, the identity
component of $H$, and we have the identification $K_0/H_0 = S^{m}$.
This implies the assertions concerning the ranks.  

 Let us now consider the long exact homotopy sequence of the   
bundle $H \to K \to S^{m}$ and deduce from it that the map $\pi_0(H)\to \pi_0(K)$ is a bijection.
This means that $H$ and $K$ have the same number of connected components, thus we may set
$\Gamma := K/K_0 = H/H_0$. 
The (free) actions of $\Gamma$ on $BK_0$ and $BH_0$ induce the identifications
$BK=BK_0/\Gamma$, respectively $BH=BH_0/\Gamma$ (see e.g.~\cite[Ch.~III, Section 1]{Hs}). Moreover,  the  map
$\pi: BH_0\to BK_0$ is $\Gamma$-equivariant. The induced homomorphism $\pi^*:H^*(BK_0)\to H^*(BH_0)$ is then $\Gamma$-equivariant as well, hence it maps $\Gamma$-invariant elements to
$\Gamma$-invariant elements. 
We have $H^*(BH)=H^*(BH_0)^\Gamma$ and $H^*(BK)=H^*(BK_0)^\Gamma$ which shows that 
if $m$ is even, then the map $\pi^*|_{H^*(BH_0)^\Gamma}: H^*(BH_0)^\Gamma\to H^*(BK_0)^\Gamma$ is injective.
We will now show that if $m$ is odd, then the latter map is surjective.
Indeed,  since $\pi^*$ is surjective, for any $b \in H^*(BK_0)^\Gamma$ there exists $a \in H^*(BH_0)$ with
$\pi^*(a)=b$. But then  $a':=\frac{1}{|\Gamma|}\sum_{\gamma\in \Gamma} \gamma a$ is in
$H^*(BH_0)^\Gamma$ and satisfies $\pi^*(a')=b$.  

\noindent{\it Case 2:} $m = 1$. We clearly have $\rk H = \rk  K -1$ in this case. We only need to show
that the map $H^*(BK)\to H^*(BH)$ is surjective. Equivalently, using the identification with the rings of invariant polynomials \cite[p.~311]{Milnor}, we will show that the restriction map $S(\mfk^*)^K\to S(\mfh^*)^H$ is surjective. Choosing an $\Ad_{K}$-invariant scalar product on $\mfk$, we obtain an orthogonal decomposition $\mfk = \mfh \oplus \RR  v$, where $v\in \mfh^\perp$. The $\Ad$-invariance implies that $\mfh$ is an ideal in $\mfk$ and $v$ a central element. Given $f\in S(\mfh^*)^H$, we define $g:\mfk=\mfh \oplus \RR v\to \RR$ by $g(X+tv)=f(X)$. Clearly, $g$ is a polynomial on $\mfk$ that restricts to $f$ on $\mfh$; for the desired surjectivity we therefore only need to show that $g$ is $K$-invariant. As $K/H \cong S^1$ is connected, $K$ is generated by its identity component $K_0$ and $H$. Note that both the $H$- and the $K_0$-action respect the decomposition $\mfk = \mfh\oplus \RR  v$. The $H$-invariance of $f$ therefore implies the $H$-invariance of $g$. Also, the adjoint action of $\mfk$ on $\mfh$ is the same as the adjoint action of $\mfh$ (the $\RR  v$-summand acts trivially), so $f$ is $K_0$-invariant, which implies that $g$ is $K_0$-invariant. 

\noindent{\it Case 3:} $m=0$. Since $\rk H = \rk K$, a maximal torus $T\subset H$ is also maximal in $K$.
The injectivity of $H^*(BK)\to H^*(BH)$ follows from the identifications $H^*(BK)=S(\mft^*)^{W(K)}$,
$H^*(BH)=S(\mft^*)^{W(H)}$.

\end{proof}

\section{Cohomogeneity one actions are Cohen-Macaulay}\label{sec:pro}
In this section we  prove Theorem  \ref{thm:cohom1eqformal}.

There are two possibilities for the orbit space $M/G$:
it can be diffeomorphic to the circle $S^1$ or to the interval $[0,1]$. 
If $M/G = S^1$,
Theorem \ref{thm:cohom1eqformal}   follows readily from \cite[Corollary 4.3]{GR}
and the fact that all isotropy groups of the $G$-action are conjugate to each other.

From now on we assume that 
 $M/G = [0,1]$. We start with some well-known considerations which hold true in this case.
One can choose a $G$-invariant Riemannian metric on $M$ and a geodesic $\gamma$ perpendicular to the orbits such that $G\gamma(0)$ and $G\gamma(1)$ are the two nonregular orbits, and such that $\gamma(t)$ is regular for all $t\in (0,1)$. Let $K^-:=G_{\gamma(0)}$, $K^+:=G_{\gamma(1)}$, and $H$ be the regular isotropy $G_\gamma$ on $\gamma$. 
We have $H \subset K^{\pm}$.  The group diagram $G \supset K^-, K^+ \supset H$ determines the equivariant diffeomorphism type of the $G$-manifold $M$. More precisely, by the slice theorem,  the boundaries of   the unit disks $D_{\pm}$ in the 
normal spaces $\nu _{\gamma(0)} G\gamma(0)$, respectively $\nu_{\gamma(1)}G\gamma(1)$ are spheres $K^{\pm}/H = S^{\ell_{\pm}}$. 
The space $M$ can be realized by gluing the tubular neighborhoods 
$G/K^{\pm}\times_{K^{\pm}} D^{\pm}$ along their common boundary 
$G/K^{\pm}\times_{K^{\pm}}K^{\pm}/H=G/H$
(see e.g.~\cite[Theorem IV.8.2]{Br}, \cite[Section 1]{GWZ} or \cite[Section 1.1]{Hoelscher}).
We are  in a position to prove the main result of the paper in the remaining case:


\begin{proof}[Proof of Theorem \ref{thm:cohom1eqformal}] \emph{in the case when $M/G=[0,1]$}. We may assume that the rank of any $G$-isotropy group
is at most equal to $b:= \rk K^-$, i.e.~ we have 
\begin{equation}\label{bi}\rk H \le \rk K^+\le b.\end{equation}  By Proposition \ref{samelson},  we have
$\rk K^- -\rk H \le 1$ and consequently $\rk H \in \{b-1, b\}$ (alternatively, we can use \cite[Lemma 1.1]{Puettmann}). If $\rk H = b$, then all isotropy groups
of the $G$-action have the same rank and Theorem \ref{thm:cohom1eqformal} follows from
\cite[Corollary 4.3]{GR}. From now on we will assume that 
\begin{equation}\label{bii}\rk H = b-1.\end{equation}
This implies that the  quotient $K^-/H$  is odd-dimensional, that is, $\ell_-$ is an odd integer.
By Proposition \ref{samelson} (b), the homomorphism 
$H^*(BK^-)\to H^*(BH)$ is surjective.
On the other hand, the Mayer-Vietoris sequence of the covering of $M$ with two tubular neighborhoods around the singular orbits 
can be expressed as follows:
\begin{equation*} 
 \ldots \longrightarrow H^*_G(M) \longrightarrow H^*_G(G/K^-)\oplus H^*_G(G/K^+) \longrightarrow H^*_G(G/H) \longrightarrow
 \ldots
\end{equation*}
But $H^*_G(G/K^{\pm})=H^*(BK^{\pm})$ and $H^*_G(G/H)=H^*(BH)$, hence the last map in the above sequence is
surjective. Thus the Mayer-Vietoris sequence splits into  short exact sequences of the form:
\begin{equation} \label{eq:MVseq}
0 \longrightarrow H^*_G(M) \longrightarrow H^*_G(G/K^-)\oplus H^*_G(G/K^+) \longrightarrow H^*_G(G/H) \longrightarrow 0.
\end{equation}
We analyze separately the  two situations imposed by equations (\ref{bi}) and (\ref{bii}).

\noindent{\it Case 1: $\rk K^+ = b$.}  By Proposition \ref{lem:cohomrank}, both $H^*_G(G/K^-)$ and
$H^*_G(G/K^+)$ are Cohen-Macaulay modules over $H^*(BG)$ of dimension $b$, so by Corollary \ref{sumofCM}, the middle term 
$ H^*_G(G/K^-)\oplus H^*_G(G/K^+)$ in the sequence \eqref{eq:MVseq} is also Cohen-Macaulay of dimension $b$.
Because $H^*_G(G/H)$ is Cohen-Macaulay of dimension $b-1$, we can apply Lemma \ref{ineq} to the sequence \eqref{eq:MVseq} to deduce that $b\leq \depth H^*_G(M)\leq \dim H^*_G(M)=b$, which implies that $H^*_G(M)$ is Cohen-Macaulay of dimension $b$.

\noindent {\it Case 2: $\rk K^+ = b-1$.} 
In this case, $H\subset K^+$ are Lie groups of equal rank
and therefore, by Proposition \ref{samelson} (a), the canonical map $H^*(BK^+)\to H^*(BH)$ is injective. Exactness of the sequence \eqref{eq:MVseq} thus implies that the restriction map $H^*_G(M) \to H^*_G(G/K^-)$ is injective. 
We deduce that the long exact sequence of the pair
$(M, G/K^-)$ splits into short exact sequences, i.e.~the following sequence is exact:
\begin{equation}\label{hgmb}
0 \longrightarrow H^*_G(M)\longrightarrow H^*_G(G/K^-) \longrightarrow H^*_G(M,G/K^-) \longrightarrow 0.
\end{equation}
We aim  to show that $H^*_G(M,G/K^-)$ is a Cohen-Macaulay module over $H^*(BG)$ of dimension $b-1$; once we have established this, we can apply Lemma \ref{ineq} to the sequence \eqref{hgmb} to deduce that $H^*_G(M)$ is Cohen-Macaulay of dimension $b$, in exactly the same way as we used the sequence \eqref{eq:MVseq} in Case 1 above.

By excision, $H^*_G(M,G/K^-)=H^*_G(G \times_{K^+} D^+, G/H)$, where $G \times_{K^+} D^+$ is a tubular neighborhood of $G/K^+$ and 
$G/H = G \times_{K^+} S^+$ is its boundary (i.e.~$S^+$ is the boundary of $D^+$). We have the isomorphism
\begin{align*}
H^*_G(G \times_{K^+} D^+,G/H) &= H^*(EG\times_G (G\times_{K^+} D^+), EG\times_G (G\times_{K^+} S^+)) \\
&= H^*(EG\times_{K^+} D^+,EG\times_{K^+} S^+)\\
& = H^*_{K^+}(D^+,S^+),
\end{align*}
which is $H^*(BG)$-linear, where the $H^*(BG)$-module structure on the right hand side is the restriction of the $H^*(BK^+)$-module structure to $H^*(BG)$. Because $H^*(BK^+)$ is a finitely generated module over $H^*(BG)$, it is sufficient to show that $H^*_{K^+}(D^+,S^+)$ is a Cohen-Macaulay module over $H^*(BK^+)$ of dimension $b-1$. Then  Lemma \ref{serre} implies the claim, because the ring $H^*(BK^+)$ is Noetherian
and *local (by \cite[Example 1.5.14 (b)]{BrunsHerzog}).

Denote by $K^+_0$ the connected component of $K^+$. There is a spectral sequence converging to $H^*_{K^+_0}(D^+,S^+)$ whose $E_2$-term equals $H^*(BK^+_0)\otimes H^*(D^+,S^+)$; because $H^*(D^+,S^+)$ is concentrated only in degree $\ell^++1$, this spectral sequence collapses at $E_2$, and it follows that $H^*_{K^+_0}(D^+,S^+)$ is a free module over $H^*(BK^+_0)$
(cf.~e.g.~\cite[Lemma C.24]{GGK}). In particular, it is a Cohen-Macaulay module over $H^*(BK^+_0)$ of dimension $b-1$. Because $H^*(BK^+_0)$ is finitely generated over $H^*(BK^+)$, Lemma \ref{serre} implies that $H^*_{K^+_0}(D^+,S^+)$ is also Cohen-Macaulay over $H^*(BK^+)$ of dimension $b-1$. 

A standard averaging argument (see for example \cite[Lemma 2.7]{GR}) shows that the $H^*(BK^+)$-module $H^*_{K^+}(D^+,S^+)=H^*_{K^+_0}(D^+,S^+)^{K^+/K^+_0}$ is a direct summand of $H^*_{K^+_0}(D^+,S^+)$: we have $H^*_{K^+_0}(D^+,S^+) = H^*_{K^+}(D^+,S^+) \oplus \ker a$, where $a:H^*_{K^+_0}(D^+,S^+)\to H^*_{K^+_0}(D^+,S^+)$ is given by averaging over the $K^+/K^+_0$-action. It follows that $H^*_{K^+}(D^+,S^+)$ is Cohen-Macaulay over $H^*(BK^+)$ of dimension $b-1$.

\end{proof}

Combining this theorem with Proposition \ref{gr}, we also immediately obtain Corollary \ref{cor:eqf}.

The following statement was shown in the above proof; we formulate it as a separate proposition because we will use it again later.
\begin{prop}\label{prop:mvseqexact} Whenever $M/G = [0,1]$ and  $\rk H = \rk K^- -1$, the sequence \eqref{eq:MVseq} is exact.
\end{prop}
An immediate corollary to this proposition is the following 
 Goresky-Kottwitz-MacPher\-son \cite{GKM} type description of the
ring $H^*_G(M)$ (we denote by ${\mathfrak k}^\pm$ and ${\mathfrak h}$
 the Lie algebras of $K^{\pm}$, respectively $H$).

\begin{cor}\label{cor:ifev}  If  $M/G=[0,1]$ and $\rk H = \rk K^- -1$  then $H^*_G(M)$ is isomorphic to the $S(\mfg^*)^G$-subalgebra of $S(({\mathfrak k}^-)^*)^{K^-}\oplus
S(({\mathfrak k}^+)^*)^{K^+}$ consisting of all  pairs $(f,g)$ with the property that $f|_{\mathfrak h}=g|_{\mathfrak h}$.
\end{cor}

We end this section with an example which concerns Corollary \ref{<2}.
It shows that if the cohomogeneity of the group action is no longer smaller than two, then the action is not necessarily
Cohen-Macaulay.
   
\begin{ex}\label{exor} The $T^2$-manifold $R(1,0)$ appears in the classification of $T^2$-actions on $4$-manifolds by Orlik and Raymond \cite{Or-Ra}. It is a compact connected 4-manifold,
homeomorphic to the connected sum $(S^1\times S^3)\# (S^2\times S^2)$. The $T^2$-action on $R(1,0)$ is effective, therefore of cohomogeneity two, and has exactly two fixed points. The action is not Cohen-Macaulay, because otherwise it would be equivariantly formal (as it has fixed points); as the fixed point set is finite, this would imply using Proposition \ref{wealso} that $H^{\rm odd}(R(1,0)) =\{0\}$, contradicting $H^1(R(1,0))=\RR$.
\end{ex}

\section{Equivariantly formal actions of cohomogeneity one}\label{sec:equiv}

In this section we present some extra results in the situation when $M/G =[0,1]$ and the action $G\times M \to M$ is equivariantly formal.
By Corollary \ref{cor:eqf}, this is equivalent to the fact that the rank of at least one of $K^-$ and $K^+$ is equal to the rank of $G$.


\subsection{The case when $M$ is even-dimensional}

\subsubsection{Cohomogeneity-one manifolds with positive Euler characteristic}\label{subsec:ec}
A discussion concerning the Euler characteristic of a compact manifold (of arbitrary dimension) admitting a cohomogeneity one action of
a compact connected Lie group  
can be found in \cite[Section 1.2]{AlekseevskyPodesta}
(see also \cite[Section 1.3]{Fr}). By Proposition 1.2.1 therein we have 
\begin{equation}\label{eq:eulercharacteristic}
\chi(M) = \chi(G/K^-)+\chi(G/K^+)-\chi(G/H).
\end{equation}
The Euler characteristic $\chi(M)$ is always nonnegative, and one has $\chi(M)>0$ if and only if
$M/G=[0,1]$, $M$ is even-dimensional, and the rank of at least one  of $ K^-$ and $ K^+$   is equal to 
$\rk G$, see \cite[Corollary 1.2.2]{AlekseevskyPodesta}. What we can show with our methods is
that  $\chi(M)>0$ implies that  $H^{\rm odd}(M)=\{0\}$. 


\begin{prop}\label{cor:ifm} 
Assume that  a compact connected manifold $M$ admits a cohomogeneity one action of
a compact connected Lie group. Then the following  conditions
are equivalent:
\begin{enumerate}
\item[(i)] $\chi(M)>0$,
\item[(ii)] $H^{\rm odd}(M)=\{0\}$,
\item[(iii)] $M$ is even-dimensional, $M/G=[0,1]$, and the $G$-action is equivariantly formal.
\end{enumerate}
These conditions  imply that
\begin{equation}\label{hstar}\dim H^*(M) = \chi(M) =\chi(G/K^-) + \chi(G/K^+).
\end{equation}
\end{prop}
\begin{proof} The equivalence of (i) and (iii) follows from the previous considerations and Corollary
\ref{cor:eqf}.
Let us now assume that $\chi(M) >0$. Then (iii) holds:
 consequently, $\rk H = \rk G-1$ and the space $\Mmax$ is either $G/K^-$, $G/K^+$ or their union $G/K^-\cup G/K^+$. A maximal torus $T\subset G$ therefore acts on $M$ with fixed point set $M^T$ equal to $(G/K^-)^T$, $(G/K^+)^T$ or $(G/K^-)^T\cup (G/K^+)^T$, depending on the rank of $K^-$ and $K^+$. As $M^T$ is in particular finite, the equivariant formality of the $G$-action is the same as the condition $H^{\rm odd}(M)=\{0\}$ (see Proposition \ref{wealso}), and
 (ii) follows. 
 The last assertion in the corollary 
 follows readily from Equation \eqref{eq:eulercharacteristic}. 
\end{proof}

\begin{rem}\label{r1}   
  Note that by the classical result of Hopf and Samelson \cite{Ho-Sa} concerning the Euler-Poincar\'e characteristic of a compact
  homogeneous space, along with the formula of  Borel given by Equation \eqref{borel} below, the topological obstruction given by the equivalence of (i) and (ii) holds true also for homogeneity, i.e., the existence of a transitive action of a compact Lie group.  
  
  We remark that the equivalence of (i) and (ii) is no
  longer true in the case when $M$ only admits an action of cohomogeneity two. Consider for example the cohomogeneity two $T^2$-manifold $R(1,0)$ 
  mentioned in Example \ref{exor}: since it is homeomorphic to
   $(S^1 \times S^3) \# (S^2 \times S^2)$, its Euler characteristic is equal to $2$, and
  the first cohomology group is $\RR$. 
\end{rem}

\begin{rem} The Euler characteristics of the nonregular orbits appearing in Equation \eqref{hstar} can also be expressed in terms of the occurring Weyl groups: we have $\chi(G/K^\pm)\neq 0$ if and only if $\rk K^\pm = \rk G$, and in this case $\chi(G/K^\pm)=\frac{|W(G)|}{|W(K^\pm)|}$. 
\end{rem}

\begin{rem}\label{rii} By a theorem of Grove and Halperin \cite{GH}, if $M$ admits a cohomogeneity-one action, then
the rational homotopy  $\pi_*(M)\otimes \QQ$ is finite-dimensional. This result also implies the equivalence between $\chi(M)>0$ and $H^{\rm odd}(M)=0$, as follows: 
If $\chi(M)>0$, then one can show using the Seifert-van Kampen theorem that $\pi_1(M)$ is finite, see Lemma \ref{pi1finite} below. Therefore the universal cover of $M$, call it $\widetilde{M}$, is also compact, and carries a cohomogeneity one action with orbit space $[0,1]$. The theorem of Grove and Halperin then states that  $\widetilde{M}$ is rationally elliptic.
Since the covering $\widetilde{M}\to M$ is finite, we have $\chi(\widetilde{M})>0$, and by \cite[Corollary 1]{Halperin} (or \cite[Proposition 32.16]{FHT}), this implies $H^{\rm odd}(\widetilde{M})=\{0\}$.
But $M$ is the quotient of $\widetilde{M}$ by a finite group, say $\Gamma$, hence
$H^{\rm odd}(M)=H^{\rm odd}(\widetilde{M})^\Gamma =\{0\}.$ 

Note that by Proposition \ref{wealso} this line of argument also implies that the $G$-action is equivariantly formal, and hence provides an alternative proof of Corollary \ref{cor:eqf} in the case of an even-dimensional cohomogeneity-one manifold $M$ with $M/G=[0,1]$.
\end{rem} 

\begin{lem}\label{pi1finite}
Assume that  a compact connected manifold $M$ admits a cohomogeneity one action of
a compact connected Lie group. If $\chi(M) >0$  then $\pi_1(M)$ is finite.
\end{lem}
\begin{proof} We may assume that $\rk K^- = \rk G$.
Consider the tubular neighborhoods of the two non-regular orbits
$G/K^+$, respectively $G/K^-$ mentioned before: their union is the whole $M$ and their intersection is $G$-homotopic to a principal orbit $G/H$.
The inclusions of the intersection into each of the two neighborhoods 
induce between the first homotopy groups the same maps as those  induced 
by the canonical projections $\rho^+:G/H\to G/K^+$, respectively 
$\rho^-:G/H \to G/K^-$ (see \cite[Section 1.1]{Hoelscher}). These are the maps $\rho^+_*:\pi_1(G/H)\to \pi_1(G/K^+)$ and
$\rho^-_* : \pi_1(G/H) \to \pi_1(G/K^-)$. Consider the bundle
$G/H\to G/K^-$ whose fiber is  $K^-/H=S^{\ell_-}$
of odd dimension $\ell_- =\dim K^- - \dim H \ge 1$.   The long exact homotopy sequence
of this bundle implies readily that the map $\rho^-_*$ is surjective.
From the Seifert-van Kampen theorem we deduce that $\pi_1(M)$ is isomorphic
to $\pi_1(G/K^+)/A$, where $A$ is the smallest normal subgroup of
$\pi_1(G/K^+)$ which contains $\rho^+_*(\ker \rho^-_*)$
(see e.g.~\cite[Exercise 2, p.~433]{Mu}).

\noindent {\it Case 1: $\dim K^+-\dim H \ge 1$.} As above, this implies
that $\rho_*^+$ is surjective. Consequently, the map
$\pi_1(G/K^-)=\pi_1(G/H)/\ker \rho_*^- \to \pi_1(G/K^+)/A$ induced
by $\rho^+_*$ is surjective as well. On the other hand, 
$\rk G = \rk K^-$ implies that $\pi_1(G/K^-)$ is a finite group.
Thus, $\pi_1(G/K^+)/A$ is a finite group as well.

\noindent {\it Case 2: $\dim K^+=\dim H$.} We have $K^+/H=S^0$, which
consists of two points, thus $\rho^+$ is a double covering.
This implies that $\rho^+_*$ is injective and its image,
$\rho^+_*(\pi_1(G/H))$, is a subgroup of index two in $\pi_1(G/K^+)$.
The index of $\rho^+_*(\ker \rho^-_*)$ in $\rho^+_*(\pi_1(G/H))$ is equal to the
index of $\ker \rho^-_*$ in $\pi_1(G/H)$, which is finite (being equal to the cardinality
of $\pi_1(G/K^-)$). Thus, the quotient $\pi_1(G/K^+)/\rho^+_*(\ker \rho^-_*)$
is a finite set. Finally, we only need to take into account that the canonical projection map $\pi_1(G/K^+)/\rho^+_*(\ker \rho^-_*)\to \pi_1(G/K^+)/A$ is surjective.
\end{proof}

\begin{rem}
Assume again that $M/G=[0,1]$ and $M$ is even-dimensional. 
If $M$ admits a  metric of positive sectional curvature, then at least one of $K^-$ and $K^+$ has maximal rank, by the so-called ``Rank Lemma" (see \cite[Lemma 2.5]{Grove} or \cite[Lemma 2.1]{GWZ}). By Corollary  \ref{cor:eqf} and Proposition \ref{cor:ifm}, the $G$-action is equivariantly formal and $H^{\rm odd}(M)=\{0\}$. This however is not a new result as by Verdiani \cite{Verdiani}, $M$ is already covered by a rank one symmetric space. 
\end{rem}

\subsubsection{Cohomology}
Consider again the situation that $M$ is a compact even-dimen\-sional manifold admitting a cohomogeneity one action of a compact connected Lie group $G$ such that $M/G=[0,1]$ and that at least one isotropy rank equals the rank of $G$. In this section we will give a complete description of the Poincar\'e polynomial of $M$, purely in terms of $G$ and the occurring isotropy groups, and eventually even of the ring $H^*(M)$.

\begin{prop}\label{cor:cohom} If $M$ is even-dimensional, $M/G=[0,1]$,  and the
rank of at least one of $K^-$ and $K^+$ equals the rank of $G$, 
then the Poincar\'e polynomial of $M$ is given by
$$P_t(M) = \frac{P_t(BK^-) + P_t(BK^+) - P_t(BH)}{P_t(BG)}.$$
In particular, $P_t(M)$ only depends on the abstract Lie groups $G,K^\pm,H$, and not on the whole group diagram.
\end{prop}

\begin{proof}  Assume that $\rk K^-=\rk G$. Since the pincipal orbit $G/H$ is odd-dimensional, 
and  $\rk H\in \{\rk K^-, \rk K^--1\}$ (see Proposition \ref{samelson}), we actually have $\rk H = \rk K^- -1$. Hence by Proposition \ref{prop:mvseqexact}, the Mayer-Vietoris sequence \eqref{eq:MVseq} is exact, which implies  that the $G$-equivariant Poincar\'e series of
$M$ is
$$P_t^G(M)= P_t^G(G/K^-) + P_t^G(G/K^+) - P_t^G(G/H)
        = P_t(BK^-) + P_t(BK^+) - P_t(BH).$$
We only need to observe that, since the $G$-action on $M$ is equivariantly formal, 
Proposition \ref{abo} (b) implies that
        $P_t^G(M) = P_t(M) \cdot P_t(BG)$.
\end{proof}

\begin{rem}\label{rem:PoincareSeriesSimplified} In  case  $H$ is connected, the  above description of $P_t(M)$ can be simplified 
as follows. Assume that $\rk K^- = \rk G$. Then $\rk H = \rk K^- - 1$, and we have that $K^-/H= S^{\ell_-}$ is an odd-dimensional sphere. Consequently, the Gysin sequence of
the spherical bundle $K^-/H\to BH \to BK^-$ splits   into short exact sequences, which
implies readily that $P_t(BH)=(1-t^{\ell_- +1})P_t(BK^-)$.
Similarly, if also the rank of $K^+$ is equal to the rank of $G$, then
$P_t(BH)=(1-t^{\ell_++1})P_t(BK^+)$ and we obtain the following formula:
$$P_t(M) = \left(\frac{1}{1-t^{\ell_-+1}}+\frac{1}{1-t^{\ell_++1}} -1\right)\cdot \frac{P_t(BH)}{P_t(BG)}.$$  
\end{rem}

Numerous examples of cohomogeneity one actions which satisfy the hypotheses 
of Proposition \ref{cor:cohom} can be found in  \cite{AlekseevskyPodesta} and
\cite{Fr} (since the condition on the isotropy ranks in Proposition \ref{cor:cohom} is equivalent  to $\chi(M)>0$, see Proposition \ref{cor:ifm}). 
Proposition \ref{cor:cohom} allows us to calculate the cohomology groups of
$M$ in all these examples. We will do in detail one such example:

\begin{ex} (\cite[Table 4.2, line 5]{AlekseevskyPodesta}) We have $G=SO(2n+1)$, 
$K^-= SO(2n)$, $K^+=SO(2n-1)\times SO(2)$,
and $H=SO(2n-1)$. The following can be found in \cite[Ch.~III, Theorem 
3.19]{Mi-To}:
\begin{align*}
{}&P_t(BSO(2n+1))= \frac{1}{(1-t^4)(1-t^8) \cdots (1-t^{4n})}.
\end{align*}
We have $\ell_- = 2n-1$ and $\ell_+=1$, and consequently, using Remark \ref{rem:PoincareSeriesSimplified}, we obtain the following description of the Poincar\'e
series of the corresponding manifold $M$:  
\begin{align*}
P_t(M)&=\left(\frac1{1-t^{2n}} + \frac1{1-t^2} -1\right) \cdot (1-t^{4n}) \\
&=1 + t^2 + t^4 + \ldots  + t^{2n-2} + 2t^{2n} + t^{2n+2} + \ldots + t^{4n}.
\end{align*}
\end{ex}

Let us now use equivariant cohomology to determine the ring structure of $H^*(M)$ in the case at hand. In Corollary \ref{cor:ifev} we determined the $S(\mfg^*)^G$-algebra structure of $H^*_G(M)$, and because of  Proposition \ref{abo} (c), the ring structure of $H^*(M)$ is  encoded in the $S(\mfg^*)^G$-algebra structure. The following proposition follows immediately.
\begin{prop}\label{prop:ringstructure} If $M$ is even-dimensional, $M/G=[0,1]$, and the rank of at least one of $K^-$ and $K^+$ equals the rank of $G$, then we have a ring isomorphism
\[
H^*(M) \simeq {\mathbb R}\otimes_{S(\mfg^*)^G}A,
\]
where $A$ is the $S(\mfg^*)^G$-subalgebra of $S((\mfk^-)^*)^{K^+}\oplus S((\mfk^+)^*)^{K^+}$ consisting of all pairs $(f,g)$ with the property that $\left.f\right|_{\mfh}=\left.g\right|_{\mfh}$.
\end{prop}

\begin{rem}
Recall that the cohomology ring of a homogeneous space $G/K$ where $G, K$ are compact and connected, 
such that $\rk G = \rk K$,
is described by Borel's formula \cite{Bo} as follows:
\begin{equation}\label{borel} H^*(G/K) \simeq {\mathbb R}\otimes_{S(\mfg^*)^G} S(\mfk^*)^K.\end{equation} 
The above description of the ring $H^*(M)$ can be considered as
a version of   Borel's formula for cohomogeneity one manifolds.
\end{rem}
Equation (\ref{borel}) is particularly simple in the case when $K=T$, a maximal torus in $G$.
Namely, if ${\mathfrak t}$ is the Lie algebra of $T$ and $W(G)$ the Weyl group of $G$, then
$$H^*(G/T) \simeq S({\mathfrak t}^*)/\langle S({\mathfrak t}^*)^{W(G)}_{>0}\rangle,$$
where $\langle S({\mathfrak t}^*)^{W(G)}_{>0}\rangle$ is the ideal of $S({\mathfrak t}^*)$ generated by the non-constant $W(G)$-invariant polynomials. The following example describes the cohomology ring of a space that can be considered
the cohomogeneity one analogue of $G/T$. It also shows that the ring $H^*(M)$ depends on the group diagram, not only on the isomorphism types
of  $G$, $K^\pm$, and $H$, like the Poincar\'e series, see Proposition \ref{cor:cohom} above.

\begin{ex} Let $G$ be a compact connected Lie group, $T$  a maximal torus in $G$,
 and $H\subset T$ a codimension one subtorus.
The cohomogeneity one manifold corresponding to $G$, $K^-=K^+:=T$, and $H$ is
$M=G\times_T S^2$, where the action of $T$ on $S^2$ 
is determined by the fact that  $H$ acts trivially and $T/H$ acts in the standard way, via rotation
about a diameter of $S^2$: indeed,   the latter $T$-action has the orbit space equal to
$[0,1]$, the singular isotropy groups both equal to $T$, and the regular isotropy group
equal to $H$; one uses \cite[Proposition 1.6]{Hoelscher}.  
Let $\mathfrak{h}$ be the Lie algebra of $H$ and pick $v\in \mathfrak{t}$ such that $\mathfrak{t}=\mathfrak{h} \oplus \RR v$.
Consider the linear function $\alpha : \mathfrak{t} \to \RR$ along with the action of $\ZZ_2=\{1,-1\}$ on $\mathfrak{t}$
given by
$$ \alpha(w+rv) =r, \quad (-1).(w+rv)=w-rv,$$ 
for all $w\in \h$ and $r\in \RR$. We denote the induced $\ZZ_2$-action on $S(\mathfrak{t}^*)$ by
$(-1).f=:\tilde{f},$ for all $f\in S(\mathfrak{t}^*)$.  
Corollary \ref{cor:ifev}  induces the $H^*(BG)=S(\mft^*)^{W(G)}$-algebra isomorphism
$$H^*_G(M) \simeq \{(f,g) \in S(\mathfrak{t}^*)\oplus S(\mathfrak{t}^*) \, : \,  \alpha \ {\rm divides} \ f-g\},$$
and the right hand side is, as an $S(\mft^*)$-algebra, isomorphic to $H^*_T(S^2)$, where the $T$-action on $S^2$ is the one described above. Note that $T/H$ can be embedded as a maximal torus in $SO(3)$, in such a way that the latter group acts canonically on $S^2$ and induces the identification $S^2= (H\times SO(3))/T$. We apply \cite[Theorem 2.6]{GHZ} for this homogeneous space  and deduce that the map $S(\mathfrak{t}^*)\otimes_{S(\mathfrak{t}^*)^{\ZZ_2}}S(\mathfrak{t}^*)\to H^*_G(M)$ given by
$f_1\otimes f_2\mapsto (f_1f_2, f_1\tilde{f}_2)$ is an isomorphism of $S(\mft^*)^{W(G)}$-algebras,
where the structure of $S(\mft^*)^{W(G)}$-algebra on $S(\mathfrak{t}^*)\otimes_{S(\mathfrak{t}^*)^{\ZZ_2}}S(\mathfrak{t}^*)$
is given by inclusion into the first factor. 
By Corollary \ref{cor:ifev} (b), we have the ring isomorphism
$$H^*(M) \simeq \RR \otimes_{S({\mathfrak t}^*)^{W(G)}}(S(\mathfrak{t}^*)\otimes_{S(\mathfrak{t}^*)^{\ZZ_2}}S(\mathfrak{t}^*))
= \left(S({\mathfrak t}^*)/\langle S({\mathfrak t}^*)^{W(G)}_{>0}\rangle\right)\otimes_{S(\mathfrak{t}^*)^{\ZZ_2}}S(\mathfrak{t}^*)
.$$
To obtain descriptions in terms of generators and relations  we need an extra variable $u$ with $\deg u =2$ and also a set of
Chevalley generators \cite{Che} of $S({\mathfrak t}^*)^W$, call them $f_1, \ldots, f_k$, where $k:=\rk G$.
We have: 
$$H^*_G(M)=S({\mathfrak t}^*)\otimes \RR[u]/\langle \alpha^2-u^2\rangle, \quad
H^*(M)=S({\mathfrak t}^*)\otimes \RR[u]/\langle  \alpha^2-u^2, f_1, \ldots, f_k\rangle.$$ 
Let us now consider two concrete situations. For both of them we have  $G=U(3)$ and $T$ is the space of all diagonal matrices in $U(3)$; the role of $H$ is played by the subgroup $\{ {\rm Diag}(1, z_2, z_3) \,: \, |z_2|=|z_3| =1\}$ in the first situation,
respectively  $\{ {\rm Diag}(z_1, z_2, z_3) \, : \, |z_1|=|z_2|=|z_3| =1, z_1z_2z_3=1\}$ in the second.
If we denote the corresponding manifolds by $M_1$ and $M_2$, then
\begin{align*}
{}& H^*(M_1)\simeq \RR[x_1, x_2, x_3,u]/\langle x_1^2-u^2,  x_1+x_2+x_3, x_1x_2+x_1x_3+x_2x_3, x_1x_2x_3\rangle \\
{\rm and} \\
{}& H^*(M_2)\simeq \RR[x_1, x_2, x_3,u]/\langle u^2, x_1+x_2+x_3, x_1x_2+x_1x_3+x_2x_3, x_1x_2x_3\rangle.
 \end{align*}
We observe that, even though $H^*(M_1)$ and $H^*(M_2)$ are isomorphic as groups, by
Proposition \ref{cor:cohom}, they are not isomorphic as rings. Indeed, we  write
\begin{align*}
& H^*(M_1)\simeq \RR[x_1, x_2,u]/\langle x_1^2-u^2,   x_1^2+x_2^2+x_1x_2, x_1x_2(x_1+x_2)\rangle \\
{\rm and}& {} &{} & {} &{} & {} &{} & {} &{} & {} &{} & {} &{} & {} &{}   \\
& H^*(M_2)\simeq \RR[x_1, x_2, u]/\langle u^2, x_1^2+x_2^2+x_1x_2, x_1x_2(x_1+x_2)\rangle
 \end{align*}
 and  note that $H^*(M_1)$ does not contain an element of degree $2$ and order $2$.
\end{ex}

\subsection{The case when  $M$ is odd-dimensional}\label{subsec:misod}
Assume we are given a cohomogeneity one action of a compact connected Lie group $G$ on a compact connected odd-dimensional manifold $M$ such that $M/G=[0,1]$, with at least one isotropy group of maximal rank.
This time the principal orbit $G/H$ is even-dimensional,
hence  the ranks of $G$ and $H$ are congruent modulo $2$. By Proposition \ref{samelson}, the rank of $H$ differs from the singular isotropy ranks by only at most one; consequently, $ \rk H=\rk G$, i.e.~all isotropy groups of the
$G$-action have maximal rank. (Note also that conversely, $\rk H = \rk G$ implies that $M$ is odd-dimensional.)

The applications we will present here concern the Weyl group associated to the cohomogeneity one action of $G$.
To define it, we first pick a $G$-invariant metric on $M$. The corresponding Weyl group $W$ is 
defined as the $G$-stabilizer of the geodesic $\gamma$ modulo $H$.  We have that $W$ is a dihedral group generated by the symmetries of
$\gamma$ at $\gamma(0)$ and $\gamma(1)$. It is finite if and only if $\gamma$ is closed. (For more on this notion, see \cite[Section 1]{GWZ}, \cite[Section 1]{Ziller},
and the references therein.)
In the case addressed in this section we have the following relation between the dimension of the cohomology of $M$ and the order of the occuring Weyl groups:
\begin{cor}\label{cor:oddd} If  $M/G = [0,1]$ and $\rk H = \rk G$,  then $W$ is finite and
\[
\dim H^*(M) = 2 \cdot \frac{\chi(G/H)}{|W|}.
\]
\end{cor}
\begin{proof} Let $T$ be a maximal torus in $H$, which then is automatically a maximal torus in $K^+$, $K^-$, and $G$. We wish to understand the submanifold $M^T$ consisting of the $T$-fixed points. 
To this end, note that for each $p\in M^T$, the $T$-action on the orbit $Gp$ has isolated fixed points,
hence the $T$-action on the tangent space $T_p(Gp)$
 has no fixed vectors; 
 if $p$ is regular, then $\nu_p (Gp)$ is one-dimensional, hence $T$ must act 
 trivially on it. 
  Therefore, $M^T$ is a finite union of closed one-dimensional totally geodesic submanifolds of $M$, more precisely: a finite union of closed normal geodesics.
 The geodesic $\gamma$ is among them and therefore closed; hence, $W$ is finite.

The order $|W|$ of the Weyl group has the following geometric interpretation, see again 
 \cite[Section 1]{GWZ} or \cite[Section 1]{Ziller}: the image of $\gamma$ intersects the regular part of $M$ (i.e., $M\setminus (G/K^-\cup G/K^+)$) in a finite number of geodesic segments; this number equals $|W|$.  Note that each such 
 geodesic segment intersects the regular orbit $G/H$ in exactly one point,
which is obviously in $(G/H)^T$. Consequently,  $W(G)$ acts transitively on the components of $M^T$, hence each of the components of $M^T$ meets $G/H$ equally often.
 This yields  a bijective correspondence between $(G/H)^T$ and the components of
 $M^T\setminus (G/K^-\cup G/K^+)$.  The cardinality of $(G/H)^T$ is equal to 
 $\chi(G/H)$ and we therefore have
\begin{equation}\label{eq:chi}
\chi(G/H) = |W| \cdot (\text{number of components of }M^T).
\end{equation}
As each component of $M^T$ is a circle and therefore contributes with $2$ to $\dim H^*(M^T)$, we obtain from Proposition \ref{abo} (e) that
\begin{align*}
\dim H^*(M) &= \dim H^*(M^T) \\
& = 2\cdot (\text{number of components of }M^T) \\
&= 2 \cdot \frac{\chi(G/H)}{|W|}.
\end{align*}
\end{proof}

\begin{rem} If   $G\times M \to M$ is a cohomogeneity one action, then different choices of
$G$-invariant metrics on $M$ generally induce different Weyl groups (see for instance \cite[Section 1]{GWZ}).
However, if the action satisfies the extra hypothesis in Corollary \ref{cor:oddd}, then the Weyl group depends only
on the rational cohomology type of $M$ and the principal orbit type of the action. In particular, $W$ 
is independent on the choice of the metric, but this is merely due to the fact that even the normal geodesics do not depend on the chosen $G$-invariant metric (being just the components of $M^T$, 
as  the proof of Corollary \ref{cor:oddd} above shows).
\end{rem}

\begin{rem} In the special situation when $M$ has the rational cohomology of a product of pairwise different spheres,
the result stated in Corollary \ref{cor:oddd} has been proved by P\"uttmann, see \cite[Corollary 2]{Pu2}. 
\end{rem}

Let us now consider the case when 
$M$ is a  rational homology sphere. Examples of cohomogeneity one actions
on such spaces can be found in \cite{GWZ} (see particularly Table E, for linear actions, and Table A for actions on the Berger space $B^7=SO(5)/SO(3)$ and the seven-dimensional spaces  $P_k$);
the Brieskorn manifold $W^{2n-1}(d)$ with $d$ odd and the $SO(2)\times SO(n)$ action defined 
in \cite {HH} is also an example; finally, several of the 7-dimensional $\ZZ_2$-homology spheres that appear in the classification of cohomogeneity one actions on $\ZZ_2$-homology spheres by Asoh \cite{As, As1} are also
rational homology spheres. As usual in this section, we assume that $\rk H = \rk G$.  Under these hypotheses, Corollary \ref{cor:oddd} implies that
 $|W|=\chi(G/H)$. In fact, the latter equation holds under the (seemingly) weaker assumption that
 the codimensions of both singular orbits are odd, as it has been observed in \cite[Section 1]{Pu2}.
Combining this result with Corollary \ref{cor:oddd} we have:

\begin{cor}\label{cor:kor} Let $G$ act on $M$ with cohomogeneity one, such that $M/G = [0,1]$ and let $H$ be a principal isotropy group.
The following statements are equivalent:

(i) $M$ is a rational homology sphere and $\rk H = \rk G$.

(ii) $M$ is a rational homology sphere and the codimensions of both singular orbits are odd.

(iii)  $|W|=\chi(G/H)$.

In any of these cases, the dimension of $M$ is odd and the $G$-action is equivariantly formal.  
\end{cor}

Let us now observe that for a general cohomogeneity one action with
$M/G = [0,1]$, the Weyl group is contained in $N(H)/H$. 
The previous results allow us to make some considerations concerning whether the extra assumption $\rk G = \rk H$  
implies that   $W=N(H)/H$.  The following example shows that this is not always the case.

\begin{ex} Let us consider the cohomogeneity one action determined by
$G= SU(3)$, $K^-=K^+=S(U(2)\times U(1))$, and $H=T$, a maximal torus in $K^\pm$.
Hoelscher \cite{Ho2} denoted this manifold by $N^7_G$ and showed that it has the same integral
homology as ${\mathbb C}P^2 \times S^3$. This implies that $\dim H^*(N^7_G)= 6$. 
On the other hand, $\chi(G/H) = \chi(SU(3)/T)=6$, hence by Corollary \ref{cor:oddd}, $|W|=2$.
Since $N(H)/H=N(T)/T=W(SU(3))$ has six elements, we have
$W\neq N(H)/H$.
\end{ex} 
 
 However, it has been observed in \cite[Section 5]{GWZ} that for linear cohomogeneity one actions on spheres,
 the condition
 $\rk G = \rk H$  does imply that $W=N(H)/H$. 
The  following more general result is an observation made by P\"uttmann, see \cite[Footnote p.~226]{Pu2}.
We found it appropriate to include a proof of it, which is based on arguments already presented here.

\begin{prop} {\rm (P\"uttmann)} \label{letg} Let $G$ act on $M$ with cohomogeneity one, such that $M/G = [0,1]$ and let $H$ be a principal isotropy group.   
If $\chi(G/H)=|W|$ then $W=N(H)/H$.
\end{prop}

\begin{proof} We only need to show that $N(H)/H$ is contained in $W$.
The hypothesis $\chi(G/H) >0$ implies that $\rk G = \rk H$, hence $G/H$ is even-dimensional and
$M$ is odd-dimensional. Let $T$ be a maximal torus in $H$. 
By Equation (\ref{eq:chi}), $M^T$ consists of a single closed geodesic $\gamma$.
The group $H$ fixes the geodesic $\gamma$ pointwise.  
Consequently, any element of $N(H)$ leaves $M^T$ invariant, hence its
coset modulo $H$ is an element of $W$. 
\end{proof}

The following  example shows that the condition $\chi(G/H)=|W|$ is stronger than $W=N(H)/H$.

\begin{ex} In general, $\rk G = \rk H$ and $W=N(H)/H$ do not imply that $\chi(G/H)=|W|$. 
To see this we consider the cohomogeneity one manifold determined by
$G=Sp(2)$, $K^-=K^+= Sp(1) \times Sp(1)$, $H=Sp(1)\times SO(2)$.
Its integer homology has been determined in \cite[Section 2.9]{Ho2}:
the manifold, denoted there by $N^7_I$, has the same homology as the product $S^3\times S^4$.  
Consequently, $\dim H^*(N^7_I) = 4$. An easy calculation shows that
$\chi(G/H) = |W(Sp(2))|/(|W(Sp(1))|\cdot |W(SO(2))|) = 4$. By Corollary \ref{cor:oddd}, $|W|=2$,
hence $\chi(G/H)\neq |W|$. On the other hand, the group  $N(H)$ can be determined explicitly
as follows. We regard $Sp(2)$ as the set of all $2\times 2$ quaternionic matrices $A$ with the property that $A\cdot A^*=I_2$
and $H=Sp(1) \times SO(2)$ as the subset consisting of all diagonal matrices ${\rm Diag}(q, z)$
with $q\in {\mathbb H}$, $z\in {\mathbb C}$, $|q|=|z| =1$. Then $N(H)$ is the set of all diagonal
matrices ${\rm Diag}(q,r)$, where $q, r\in {\mathbb H}$, $|q|=|r|=1$, $r\in N_{Sp(1)}(SO(2))$.
This implies that $N(H)/H \simeq N_{Sp(1)}(SO(2))/SO(2) =W(Sp(1))={\mathbb Z}_2$, hence
$N(H)/H=W$.
\end{ex}

\subsection{Torsion in the equivariant cohomology of cohomogeneity one actions}  \label{torsfre}
It is known \cite{Fr-Pu2} that if $G\times M \to M$ is an arbitrary group action, then the
$H^*(BG)$-module $H^*_G(M)$ may be torsion-free without being free. (But note that this cannot happen if $G$ is either the circle $S^1$ or the two-dimensional torus $T^2$, see \cite{Al}.) The goal of this subsection is to show that this
phenomenon also
cannot occur if the cohomogeneity of the action is equal to one. This is a consequence of the following lemma.

\begin{lem} Assume that the action of the compact connected Lie group $G$ on the compact connected manifold $M$ 
is Cohen-Macaulay. Then the $H^*(BG)$-module $H^*_G(M)$ is free if and only if it is torsion-free.
\end{lem} 

\begin{proof} If $H^*_G(M)$ is a torsion-free $H^*(BG)$-module, then, by \cite[Theorem 3.9 (2)]{GR},
the space $\Mmax$ is non-empty. Since the $G$-action is Cohen-Macaulay, it must be equivariantly formal by Proposition \ref{gr}.
\end{proof}

From Theorem \ref{thm:cohom1eqformal} we deduce:

\begin{cor} Assume that the action $G\times M \to M$ has cohomogeneity one.
Then  the $H^*(BG)$-module $H^*_G(M)$ is free if and only if it is torsion-free.
\end{cor}

\end{document}